\documentclass[11pt,a4paper]{amsart}
\usepackage[T1]{fontenc}
\usepackage{lmodern, color}
\usepackage{amsmath,amssymb}
\usepackage[margin=30mm]{geometry}
\usepackage{microtype,booktabs}
\numberwithin{equation}{section}
\newtheorem{theorem}{Theorem}[section]
\newtheorem{lemma}[theorem]{Lemma}

\definecolor{DarkBlue}{rgb}{0,0.1,0.5}
\usepackage[
colorlinks=true,
linkcolor=DarkBlue,
citecolor=DarkBlue,
urlcolor=DarkBlue 
]{hyperref}

\newcommand{\R}{\mathbb R}

\newcommand{\Sphere}{\mathbb S}

\newcommand{\Ric}{\operatorname{Ric}}
\newcommand{\dd}{\,\mathrm d}
\DeclareMathOperator{\Span}{span}

\newcommand\nb{\nabla}
\newcommand{\beq}{\begin{equation} \begin{split}}
		\newcommand{\eeq}{\end{split} \end{equation}}

\makeatletter
\def\section{\@startsection{section}{1}\z@{.9\linespacing\@plus\linespacing}%
	{.7\linespacing} {\fontsize{13}{14}\selectfont\bfseries\centering}}
\def\paragraph{\@startsection{paragraph}{4}%
	\z@{0.3em}{-.5em}%
	{$\bullet$ \ \normalfont\itshape}}

\@addtoreset{equation}{section}
\makeatother

\renewcommand\and{\qquad\text{and}\qquad}

\newcommand\sm{\setminus}

\newcommand\Lm{\Lambda}

\newcommand{\comm}[1]{}

\def\bm1{\mathbbm{1}}

\def\arr{\rightarrow}

\def\tt{\theta}

\def\lm{\lambda}

\def\dd{{\,\mathrm{d}}}

\newcounter{counter_a}

\usepackage[latin1]{inputenc}
\usepackage[T1]{fontenc}

\newcommand{\ie}{{\it i.e.}\,}
\newcommand{\cf}{{\it cf.}\,}

\numberwithin{figure}{section}
\numberwithin{equation}{section}
\theoremstyle{plain}
\newtheorem*{thm*}{Theorem}

\theoremstyle{remark}

\theoremstyle{plain}

\newcommand{\beu}{\begin{equation*}}
	\newcommand{\eeu}{\end{equation*}}
\newcommand{\besu}{\begin{equation*}
		\begin{aligned}}
		\newcommand{\eesu}{\end{aligned}
\end{equation*}}
\newcommand{\bes}{\begin{equation}
		\begin{aligned}}
		\newcommand{\ees}{\end{aligned}
\end{equation}}

\newcommand\eps{\varepsilon}

\newcommand\void[1]{}

\def\eps{\varepsilon}

   \def\dH{{\mathbb H}}   
      
   \def\dN{{\mathbb N}}   
      \def\dR{{\mathbb R}}
\def\dS{{\mathbb S}}

   \def\sfE{{\mathsf E}}

      \def\sfR{{\mathsf R}}

\newcommand{\Hm}[1]{\leavevmode{\marginpar{\tiny%
			$\hbox to 0mm{\hspace*{-0.5mm}$\leftarrow$\hss}%
			\vcenter{\vrule depth 0.1mm height 0.1mm width \the\marginparwidth}%
			\hbox to
			0mm{\hss$\rightarrow$\hspace*{-0.5mm}}$\\
			\relax\raggedright #1}}}

\title{A note on the failure of Payne's inequality on manifolds with boundary}

\author[D.\ Buoso]{Davide Buoso }
\address{(D.\ Buoso) Dipartimento per lo Sviluppo Sostenibile e la Transizione Ecologica, Universit\`a degli Studi del Piemonte Orientale ``A.\ Avogadro'', Piazza Sant'Eusebio 5, 13100 Vercelli, Italy}
\email{davide.buoso@uniupo.it}

\author[V.\ Lotoreichik]{Vladimir Lotoreichik}
\address{(V.\ Lotoreichik) Department of Mathematics, Faculty of Nuclear Sciences and Physical Engineering, Czech Technical University in Prague, Trojanova 13, 120 00, Prague, Czech Republic}
\email{vladimir.lotoreichik@fjfi.cvut.cz}


\begin{document}
\begin{abstract}
We show that Payne's conjecture on the comparison between Dirichlet Laplacian eigenvalues and buckling eigenvalues does not hold on flat and
hyperbolic cylindrical surfaces and on (round) hemispheres.
\end{abstract}
\maketitle
\enlargethispage{3pt}

\section{Introduction}\label{sec:introduction}

Let $\Omega$ be a bounded domain with smooth non-empty boundary $\partial\Omega$  
in a complete smooth $d$-dimensional Riemannian manifold with $d\ge 2$. The domain $\Omega$ is equipped
with the metric $g$ inherited from the ambient manifold.  Denote by
$\lambda_k(\Omega)$ the Dirichlet eigenvalues of the positive
Laplace--Beltrami operator on $\Omega$ and by $\Lambda_k(\Omega)$ the clamped
buckling eigenvalues, both ordered increasingly and counted
with multiplicity.  We will address the validity of the inequality
\begin{equation*}
  \Lambda_k(\Omega)\geq\lambda_{k+1}(\Omega),\qquad k\in\mathbb{N},
  \tag{GP}\label{eq:GP}
\end{equation*}
which is known as (generalized) Payne's conjecture.
Payne \cite{Payne} proved the inequality for $k=1$ in the case of bounded planar Euclidean domains; a small gap in his
argument was later closed by Friedlander \cite{Friedlander2004},
who also showed that the stronger comparison $\Lambda_1\geq\lambda_3$
fails in general even in the case of planar Euclidean domains. Liu's proposed all-index proof \cite{Liu} was
withdrawn after an error was identified by Friedlander
\cite{Friedlander2021}. Levitin and Polterovich \cite{LevitinPolterovich}
recently proved \eqref{eq:GP} for every odd index $k$ on bounded planar domains and on certain classes of flat manifolds;
their result also covers an even index $k$ when
$\Lambda_k=\Lambda_{k+1}$, however the general even-index case remains open. 
On the other hand, Buoso, Luzzini, Provenzano, and Stubbe \cite{BLPS} showed that \eqref{eq:GP} 
fails at every even index in dimension one. Note that Payne's conjecture \eqref{eq:GP} sparkled a 
wide range of interest for comparison between eigenvalues of related yet different operators; 
see \cite{ACNT,BuosoFreitas,ElChamiEtAl,IliasShouman, Kelliher, Lotoreichik, Provenzano} 
and the references therein.

In this paper we provide some counterexamples to \eqref{eq:GP} on curved and flat
manifolds. In particular, we show that a two-dimensional flat cylindrical surface with a sufficiently short periodic direction
retains the one-dimensional failure at $k=2$. Additionally, by means of perturbation and scaling we are able to show that the two-dimensional hyperbolic cylindrical surface of every prescribed negative constant Gaussian curvature also provides a counterexample to \eqref{eq:GP}.
Moreover, by means of constructing a $(d+1)$-dimensional clamped trial space on the hemisphere $\Sphere^d_+$, we can show that
\begin{equation}
 \Lambda_{d+1}(\Sphere^d_+)
 \leq \frac{19d^2+56d+48}{7d+8}
 <3(d+2)=\lambda_{d+2}(\Sphere^d_+),\qquad d\geq2.
 \label{eq:intro-hemisphere}
\end{equation}
We remark that domains on the sphere also serve as counterexamples~\cite{Mazzeo} to the validity of Friedlander-Filonov inequality (\cf~\cite[\S 3.2.4]{LevitinMangoubiPolterovich2023}) on manifolds with boundary.

We also note that we prompted AI to conduct numerical simulations of \eqref{eq:GP} for bounded domains in $\mathbb R^d$ and $\mathbb H^d$ for $d=2,3,4,5$; the prompt given was specifically designed to search for possible counterexamples to \eqref{eq:GP}. However, after refining the numerical scheme, no counterexample was reported in those spaces. 

The paper is organized as follows: in Section~\ref{sec:framework} we provide the variational setup, and Section~\ref{sec:counterexamples} is devoted to the counterexamples to inequality \eqref{eq:GP}.


\section{Variational framework}\label{sec:framework}
Throughout the present paper, all function spaces consist of real-valued functions.
We will denote by $-\Delta_g$ the positive Laplace-Beltrami operator on $\Omega$ equipped with the metric $g$ inherited from the ambient Riemannian manifold. We recall the Dirichlet eigenvalue problem for $-\Delta_g$ on a domain $\Omega$
$$
\begin{cases}
-\Delta_g u=\lambda u, & \text{in\ }\Omega,\\
u=0, & \text{on\ }\partial\Omega,
\end{cases}
$$
which has a discrete set of positive eigenvalues $\{\lambda_k\}_{k\in\mathbb N}$, which in turn can be variationally characterized as
\begin{equation}
 \lambda_k=\min_{\substack{E\subset H^1_0(\Omega)\\ \dim E=k}}
 \max_{0\neq u\in E}\frac{\displaystyle\int_\Omega\langle\nabla u,\nabla u\rangle_g\dd V_g}{\displaystyle\int_\Omega u^2\dd V_g},
 \label{eq:dirichlet-minmax}
\end{equation}
where $\langle\cdot,\cdot\rangle_g$ stands for the inner product in the tangent space with respect to the metric $g$.
Similarly, we recall the buckling eigenvalue problem on a domain $\Omega$
\begin{equation}\label{bev}
\begin{cases}
\Delta_g^2 U=-\Lambda \Delta_g U, & \text{in\ }\Omega,\\
U=\partial_{\nu} U=0, & \text{on\ }\partial\Omega,
\end{cases}
\end{equation}
where $\partial_\nu U$ stands for the normal derivative of $U$.
This spectral problem has a discrete set of positive eigenvalues $\{\Lambda_k\}_{k\in\mathbb N}$, which in turn can be variationally characterized as
\begin{equation}
 \Lambda_k=\min_{\substack{E\subset H^2_0(\Omega)\\ \dim E=k}}
 \max_{0\neq u\in E}
 \frac{\displaystyle\int_\Omega(\Delta_g u)^2\dd V_g}{
 	\displaystyle\int_\Omega\langle\nabla u,\nabla u\rangle_g\dd V_g}. \label{eq:buckling-minmax}
\end{equation}
Here and occasionally in the following, we use the shorthand notation $\lambda_k$ and $\Lambda_k$ instead of $\lambda_k(\Omega)$ and $\Lambda_k(\Omega)$, respectively.

It is nowadays classical to show that $\Lambda_k\ge \lambda_k$ for any domain of finite measure (see e.g., \cite[Theorem 2.2]{BuosoFreitas}). If one supposes, in addition, that $\Omega$ is oriented and is bounded with a boundary of class $C^2$, then it is possible to show that in fact $\Lambda_k> \lambda_k$ (see e.g., \cite[Theorem 1.1]{Liu2011}).

We also observe that provided variational formulation for \eqref{bev} is essentially the only one in Euclidean spaces, but as soon as the curvature is non-zero then substantially different formulations may arise, since the so-called Hessian energy may differ from the Laplacian energy due to the Reilly identity~\cite[Theorem 1]{Reilly} (see also~\cite[Eq. 3.24]{ColboisProvenzano2022}) for $u\in H^2_0(\Omega)$
\begin{equation}
 \int_\Omega|D^2u|^2\dd V_g
 =\int_\Omega(\Delta_g u)^2\dd V_g
 -\int_\Omega\Ric(\nabla u,\nabla u)\dd V_g,
 \label{eq:reilly}
\end{equation}
where $D^2u$ is the Hessian of $u$, $|D^2 u|$ is its Hilbert-Schmidt norm, and $\Ric(\cdot,\cdot)$ is the Ricci tensor on $\Omega$.
For constant sectional curvature $\kappa$, $\Ric=(d-1)\kappa g$,
so min--max principle gives the exact shift
\begin{equation}
 \Lambda_k^{\mathrm{Hess}}=\Lambda_k-(d-1)\kappa,
 \label{eq:hessian-shift}
\end{equation}
where
$$
\Lambda_k^{\mathrm{Hess}}=\min_{\substack{E\subset H^2_0(\Omega)\\ \dim E=k}}
 \max_{0\neq u\in E}\frac{\displaystyle \int_\Omega|D^2u|^2\dd V_g}{\displaystyle\int_\Omega\langle\nabla u,\nabla u\rangle_g\dd V_g},
$$
Moreover, we remark that both $\lambda_k$ and $\Lambda_k$  have the same scaling law: given a constant metric rescaling $\widehat g=cg$ with $c >0$ on a fixed domain $\Omega$, we have
\begin{equation}
 \lambda_k(\Omega,\widehat g)=c^{-1}\lambda_k(\Omega,g),\qquad
 \Lambda_k(\Omega,\widehat g)=c^{-1}\Lambda_k(\Omega,g),\qquad
 K_{\widehat g}=c^{-1}K_g,
 \label{eq:scaling}
\end{equation}
where we explicitly indicated the dependence of the eigenvalues on the metric and where  by $K_g$ and $K_{\widehat g}$ we denote the sectional curvature on $\Omega$ corresponding to the metrics $g$ and $\widehat g$, respectively.

Finally, we also recall the following continuity result, which will be used later on.

\begin{lemma}\label{lem:continuity}
Suppose that the smooth metrics $g_\varepsilon$, on the closure of a fixed
bounded smooth domain $\Omega$, $C^2$-converge to the smooth metric $g_0$ as $\varepsilon\rightarrow0$. Then, for every $k\in\mathbb{N}$,
\[
 \lambda_k(\Omega,g_\varepsilon)\longrightarrow\lambda_k(\Omega,g_0),\qquad
 \Lambda_k(\Omega,g_\varepsilon)\longrightarrow\Lambda_k(\Omega,g_0),
 \qquad\text{as}~\varepsilon\rightarrow0.
 \]
\end{lemma}
\begin{proof}
	The Sobolev spaces for metrics $g_\varepsilon$ and $g_0$ are topologically equivalent. We use the same notation for these spaces  $L^2(\Omega)$, $H^1(\Omega)$ and $H^2(\Omega)$ and the respective norms $\|u\|_{L^2(\Omega)}^2 := \int_\Omega|u|^2\dd V_{g_0} $, $\|u\|_{H^1(\Omega)} ^2:= \||\nabla u|\|^2_{L^2(\Omega)}+ \|u\|^2_{L^2(\Omega)}$, and $\|u\|_{H^2(\Omega)}^2 := \||D^2u|\|^2_{L^2(\Omega)}+ \|u\|^2_{H^1(\Omega)}$ will be for convenience evaluated in the limiting metric $g_0$.
	We introduce the symmetric bilinear forms
	\[
	\begin{aligned}
		\mathfrak{d}_g[u,v]&=\int_\Omega\langle\nabla u,\nabla v\rangle_g\dd V_g,&\qquad\mathrm{dom}\,\mathfrak{d}_g = H^1_0(\Omega),\\
		\mathfrak{a}_g[u,v]&=\int_\Omega(\Delta_g u)(\Delta_g v)\dd V_g,&\qquad
		\mathrm{dom}\,\mathfrak{a}_g = H^2_0(\Omega),\\
		\mathfrak{b}_g[u,v]&=\int_\Omega uv\dd V_g,&\qquad
		\mathrm{dom}\,\mathfrak{b}_g = L^2(\Omega).
	\end{aligned}
	\]
	The Laplace-Beltrami operator in local coordinates acts as
	\[
	\Delta_g u=|g|^{-1/2}\partial_i
	\bigl(|g|^{1/2}g^{ij}\partial_j u\bigr).
	\]
	Uniform convergence of the coefficients in the expression for the Laplace-Beltrami operator
	implies that there exists $\delta_\varepsilon>0$ satisfying $\delta_\varepsilon\rightarrow0$ as $\varepsilon\rightarrow 0$ such that
	\[
	\begin{aligned}
		|\mathfrak{a}_{g_\varepsilon}[u]-\mathfrak{a}_{g_0}[u]|
		&\leq\delta_\varepsilon\|u\|_{H^2(\Omega)}^2,\qquad &\text{for all}\, u\in H^2_0(\Omega),\\
		|\mathfrak{d}_{g_\varepsilon}[u]-\mathfrak{d}_{g_0}[u]|
		&\leq\delta_\varepsilon\|u\|_{H^1(\Omega)}^2,\qquad &\text{for all}\, u\in H^1_0(\Omega),\\
		|\mathfrak{b}_{g_\varepsilon}[u]-\mathfrak{b}_{g_0}[u]|
		&\leq\delta_\varepsilon\|u\|_{L^2(\Omega)}^2,&\qquad\text{for all}\, u\in L^2(\Omega).
	\end{aligned} 
	\]
	The Dirichlet elliptic estimate and Poincar\'e inequality yield
	that there exists a constant $C \ge 1$ such that
	$\|u\|_{H^2(\Omega)}^2\leq C \mathfrak{a}_{g_0}[u]$ for all $u\in H^2_0(\Omega)$; $\|u\|_{H^1(\Omega)}^2\le C\mathfrak{d}_{g_0}[u]$ for all $u\in H^1_0(\Omega)$. Thus,  setting
	$\hat\delta_\varepsilon = C\delta_\varepsilon$, we get
	\[
	\begin{split}
		(1-\hat\delta_\varepsilon)\mathfrak{a}_{g_0}
		&\leq \mathfrak{a}_{g_\varepsilon}\leq(1+\hat\delta_\varepsilon)\mathfrak{a}_{g_0},\\
		(1-\hat\delta_\varepsilon)\mathfrak{d}_{g_0}
		&\leq \mathfrak{d}_{g_\varepsilon}\leq(1+\hat\delta_\varepsilon)\mathfrak{d}_{g_0},\\
		(1-\hat\delta_\varepsilon)\mathfrak{b}_{g_0}
		&\leq \mathfrak{b}_{g_\varepsilon}\leq(1+\hat\delta_\varepsilon)\mathfrak{b}_{g_0},
	\end{split}
	\]
	as quadratic forms. The two
	min-max principles~\eqref{eq:dirichlet-minmax} and~\eqref{eq:buckling-minmax} combined with the above estimates bound each perturbed eigenvalue corresponding to the metric $g_\varepsilon$ between factors
	tending to one times its unperturbed value corresponding to the metric $g_0$, from which the desired continuity of the eigenvalues follows.
\end{proof}


\section{Counterexamples}\label{sec:counterexamples}
In this section, we will present several counterexamples to the validity of Payne's inequality on manifolds with boundary. In particular, we will demonstrate in Subsection~\ref{ssec:cyl} that the inequality can fail on a flat cylindrical surface, in Subsection~\ref{ssec:hcyl} we will show that the inequality can also fail on a cylindrical surface with arbitrary constant negative Gaussian curvature and, finally, in Subsection~\ref{ssec:hemisphere} we will show that it also fails on a round hemisphere of any dimension. 
\subsection{Flat cylindrical surface}\label{ssec:cyl}
Let $\rho > 0$ and let $\Sphere^1_\rho$ be the circle with metric $\rho^2\dd\theta^2$,
where $0\leq\theta<2\pi$. The points on this circle can be identified with the points in the interval $[0,2\pi)$.
Define the cylindrical surface of height $1$ with the base $\Sphere^1_\rho$ as
\[
 C_\rho=(0,1)\times\Sphere^1_\rho
\]
equipped with the metric
\[
 g=\dd x^2+\rho^2\dd\theta^2.
\]
Note that the boundary components of $C_\rho$ are smooth circles.

\begin{theorem}\label{thm:flat-cylinder}
If $\rho\in(0,\frac{1}{2\sqrt{2}\pi})$, then
\[
 \Lambda_2(C_\rho)=4z_1^2<9\pi^2=\lambda_3(C_\rho),
\]
where $z_1\in(\pi,3\pi/2)$ is the first positive nonzero root of the transcendental equation
$\tan z=z$. 
\end{theorem}
\begin{proof}
By separation of variables, the Dirichlet eigenvalues of the Laplacian on $C_\rho$ can be listed as
\begin{equation}
 \lambda_{n,m}=\pi^2n^2+\frac{m^2}{\rho^2},
 \qquad n\in\mathbb{N},\quad m\in\mathbb Z.
 \label{eq:cylinder-dirichlet}
\end{equation}
Clearly, the total sequence of Dirichlet eigenvalues of the Laplacian on $C_\rho$ is the total enumeration of all the above values, where repetitions account for multiplicities.
The assumption on $\rho$ places the first non-constant angular level ($|m| = 1$)
strictly above $9\pi^2$. Thus the first three eigenvalues of the Dirichlet Laplacian
on $C_\rho$ are simple and given by $\pi^2$, $4\pi^2$, and $9\pi^2$ (all of them correspond to $m = 0$). In particular, $\lambda_3(C_\rho) = 9\pi^2$.

Again applying the separation of variables, we observe that the buckling eigenvalues on $C_\rho$ can be listed as
\begin{equation}
 \Lambda_{n,m}=\min_{\substack{E\subset H^2_0(0,1)\\ \dim E=n}}
 \max_{0\neq f\in E}\frac{\displaystyle\int_0^1\left(-f'' + \tfrac{m^2}{\rho^2}f\right)^2\dd x}{\displaystyle\int_0^1 \left((f')^2+\tfrac{m^2}{\rho^2}f^2\right)\dd x},
 \qquad n\in\mathbb{N},\quad m\in\mathbb Z.
 \label{eq:cylinder-dirichlet}
\end{equation}
The total family of buckling eigenvalues on $C_\rho$ is the enumeration of all the above values, where repetitions account for multiplicities.
The values $\Lambda_{n,0}$ for all $n\in\mathbb{N}$
can be found as values $q^2$ for $q > 0$ for which the following system
\[
\begin{cases}
    f^{(4)}+ q^2 f'' = 0,\qquad\qquad\qquad\qquad \text{in}\,\,(0,1),\\
    f(0)=f(1) =f'(0)=f'(1)=0,
\end{cases}
\]
has a non-trivial solution in $C^\infty([0,1])$. The existence of several linearly independent solutions for a given value $q^2$ accounts for the multiplicity of the respective eigenvalue.

The general solution of $f^{(4)}+q^2f''=0$ is
\[
 f(x)=a+bx+c\cos(qx)+d\sin(qx),
\]
with constants $a,b,c,d\in\mathbb{R}$.
The conditions at $x=0$ give $a=-c$ and $b=-qd$.
The remaining equations are
\[
 \begin{pmatrix}\cos q-1&\sin q-q\\-\sin q&\cos q-1\end{pmatrix}
 \begin{pmatrix}c\\d\end{pmatrix}=0.
\]
The corresponding determinant equals
\begin{equation}
 2(1-\cos q)-q\sin q
 =4\sin(q/2)\bigl[\sin(q/2)-(q/2)\cos(q/2)\bigr].
 \label{eq:cylinder-determinant}
\end{equation}
From the construction we notice that all the eigenvalues are simple.
The apparent root $q=0$ gives no nonzero clamped eigenfunction and
is excluded. The positive roots are $q=2k\pi$, $k\in\mathbb{N}$ and $q=2z_k$,
where $z_k\in(\pi k,\pi k+\pi/2)$, $k\in\mathbb{N}$ solves the transcendental equation $\tan z=z$.
On the intervals $(\pi k,\pi k+\pi/2)$ the function $\tan z-z$ increases from a negative number to
$+\infty$, with derivative $\tan^2z$; there are no other positive
roots. Therefore we get that
$\Lambda_{1,0} = 4\pi^2$ and $\Lambda_{2,0} = 4z_1^2$. Moreover, we notice that $4z_1^2<9\pi^2$.

Let $m\in\mathbb{Z}\setminus\{0\}$. Let us expand $f\in H^2_0(0,1)\setminus\{0\}$ into the series
$f(x) = \sum_{k=1}^\infty c_k\sqrt{2}\sin(\pi k x)$ with some real coefficients $\{c_k\}_{k\in\mathbb{N}}$.
This is a series in terms of normalized eigenfunctions of the Dirichlet Laplacian on the interval $(0,1)$. Since $f$ is a non-zero function, we can exclude the case when all coefficients $c_k$ are equal to zero. Then we obtain that
\[
 \frac{\displaystyle \int_0^1\left(-f''+\tfrac{m^2}{\rho^2}f\right)^2\dd x}{\displaystyle\int_0^1 \left((f')^2+\tfrac{m^2}{\rho^2}f^2\right)\dd x}=
 \frac{\displaystyle\sum_{k=1}^\infty\left(\pi^2k^2+\tfrac{m^2}{\rho^2}\right)^2|c_k|^2}
        {\displaystyle\sum_{k=1}^\infty\left(\pi^2k^2+\tfrac{m^2}{\rho^2}\right)|c_k|^2}
 \geq\pi^2+\frac{1}{\rho^2}>9\pi^2.
\]
Hence, we conclude from~\eqref{eq:cylinder-dirichlet} that $\Lambda_{1,m} > 9\pi^2$ for all $m\in\mathbb{Z}\setminus\{0\}$. Thus, $\Lambda_2(C_\rho) = 4z_1^2$ and the claim follows
from the observation that $4z_1^2 < 9\pi^2 = \lambda_3(C_\rho)$.
\end{proof}

We observe that Theorem \ref{thm:flat-cylinder} is a counterexample to \eqref{eq:GP} for $k=2$ and therefore does not violate the results in \cite{LevitinPolterovich}, where inequality \eqref{eq:GP} is shown for all flat two-dimensional cylindrical surfaces for all $k$ odd. In fact, we would like to highlight that it is possible to construct cylindrical surfaces such that inequality \eqref{eq:GP} is violated for any given $k$ even \cite{LEVpc}.

\subsection{Hyperbolic cylindrical surface}\label{ssec:hcyl}
The aim of this subsection is to show that by deformation of the previous counterexample, one can construct domains in two-dimensional manifolds with arbitrary negative Gaussian curvature, for which Payne's inequality fails.
\begin{theorem}\label{thm:hyperbolic-cylinder}
For every $\kappa<0$, there is a bounded smooth domain $\Omega$  in a two-dimensional manifold of constant Gaussian curvature $\kappa$  for which
$\Lambda_2(\Omega)<\lambda_3(\Omega)$.
\end{theorem}
\begin{proof}
On the manifold $\R\times\Sphere^1$ we define the metric
\[
 g_\varepsilon=\dd x^2+
       \rho^2\cosh^2(\sqrt{\varepsilon}\,x)\dd\theta^2,\qquad \eps > 0,
 \]
with $\rho  = \frac{1}{10} < \frac{1}{2\sqrt{2}\pi}$.
Consider now the domain $\Omega=(-1/2,1/2)\times\Sphere^1$.

For a metric $\dd x^2+f(x)^2\dd\theta^2$, the Gaussian curvature
is $K(x,\tt) = -f''(x)/f(x)$ (\cf~\cite[\S 4-6]{doCarmo1976}). Here $f''=\varepsilon f$, so $K=-\varepsilon$.
These metrics converge to the flat cylindrical surface metric in $C^\infty$
on $\overline\Omega$ as $\eps\arr0$. The strict gap in
Theorem~\ref{thm:flat-cylinder} and Lemma~\ref{lem:continuity}
imply $\Lambda_2(\Omega, g_\varepsilon)<\lambda_3(\Omega,g_\varepsilon)$
for every sufficiently small $\varepsilon>0$.

Fix one such $\varepsilon$, write $\kappa=-a$ with $a > 0$, and rescale the metric by 
\[
 \widehat g=\frac{\varepsilon}{a}g_\varepsilon.
\]
Equation~\eqref{eq:scaling} gives $K_{\widehat g}=-a=\kappa$,
while multiplying both eigenvalue sequences by $a/\varepsilon$
preserves their strict ordering. 
\end{proof}

We remark that the domains in the above theorem are not isometric to
bounded domains in $\dH^2$. Thus this construction
does not contradict the fact that we could not find numerically 
a domain in $\dH^2$ for which Payne's inequality fails.

\subsection{Round hemispheres}\label{ssec:hemisphere}
In this subsection, we show that Payne's inequality fails on a round hemisphere of any dimension.
We embed $\Sphere^d\subset\R^{d+1}$ and define
$\Sphere^d_+=\{x=(x_1,x_2,\dots,x_{d+1})\in\dS^d\colon x_{d+1}>0\}$. The equator is the smooth boundary of $\dS^d_+$.
The metric on $\Sphere^d_+$ is induced by its natural embedding in $\dR^{d+1}$.

\begin{theorem}\label{thm:hemisphere}
For every $d\geq2$,
\begin{equation}
 \Lambda_{d+1}(\Sphere^d_+)
 \leq Q_d:=\frac{19d^2+56d+48}{7d+8}
 <3(d+2)=\lambda_{d+2}(\Sphere^d_+).
 \label{eq:hemisphere-bound}
\end{equation}
The strictly positive variational margin is
\begin{equation}
 3(d+2)-Q_d=\frac{2d(d+5)}{7d+8}.
 \label{eq:hemisphere-margin}
\end{equation}
\end{theorem}
\begin{proof}
We split the argument into two steps. In the first step, we compute $\lm_{d+2}(\dS^d_+)$ and in the second step we estimate $\Lm_{d+1}(\dS^d_+)$ from above.

\noindent {\it Step 1. Dirichlet Laplacian eigenvalues on $\dS^d_+$.}
Odd reflection across the totally geodesic equator extends 
an eigenfunction of the Dirichlet Laplacian on $\dS^d_+$  to an eigenfunction of the Laplacian on $\dS^d$.
Conversely, the restrictions to $\dS^d_+$ of the eigenfunctions of the Laplacian on $\dS^d$ which are odd with respect to reflection $x_{d+1}\mapsto -x_{d+1}$ constitute the complete family of  the eigenfunctions of the Dirichlet Laplacian on $\dS^d_+$.
Below we rely on the analysis of spherical harmonics in~\cite[\S 1.2.3]{LevitinMangoubiPolterovich2023}.
 Degree $\ell\in\dN_0$ spherical harmonics correspond to the eigenvalue
$\ell(\ell+d-1)$ of the Laplacian on $\dS^d$. Using the explicit characterization of spherical harmonics
we can single out those which are odd with respect to reflection $x_{d+1}\mapsto-x_{d+1}$.
We observe that for $\ell = 0$ there are no such harmonics, for $\ell =1$
the basis of such spherical harmonics is $\{x_{d+1}\}$, while for $\ell = 2$ the basis of such spherical harmonics is $\{x_{d+1}x_1,\ldots,x_{d+1}x_d\}$.
There are non-zero odd spherical harmonics corresponding to $\ell=3$; for instance
$x_{d+1}x_1x_2$. 
Hence
\begin{equation}
 \lambda_1(\dS^d_+)=d,\qquad
 \lambda_2(\dS^d_+)=\cdots=\lambda_{d+1}(\dS^d_+)=2(d+1),\qquad
 \lambda_{d+2}(\dS^d_+)=3(d+2).
 \label{eq:hemisphere-dirichlet}
\end{equation}

\medskip

\noindent\emph{Step 2. Trial space for the clamped buckling problem.}
Take the  trial space
\[
 E=\Span\big\{x^2_{d+1},x^2_{d+1}x_1,\ldots,x^2_{d+1}x_d\big\}.
\]
These $d+1$ functions are clearly linearly independent.
Their values and gradients vanish on $\partial\dS^d_+$, so
$E\subset H^2_0(\Sphere^d_+)$. Let us denote in the following by $\Delta$ and $\nb$, respectively, the Laplacian and the gradient on $\dS^d$ and by $\langle\cdot,\cdot\rangle$ the inner product in the tangent space to $\dS^d$ with respect to the round metric.
We have standard identities for $i, j\in\{1,2,\dots,d+1\}$
\[
 \Delta x_j=-dx_j,\qquad
 \langle\nabla x_i,\nabla x_j\rangle=\delta_{ij}-x_ix_j.
\]
The product rule gives
\begin{equation}
	\label{eq:sphere-t2}
\begin{aligned}
 \Delta(x_{d+1}^2)&=2-2(d+1)x^2_{d+1},\\
 \Delta(x_{d+1}^2x_i)&=  x_i[2-2(d+1)x_{d+1}^2]-dx_{d+1}^2x_i+4x_{d+1}\langle\nabla x_{d+1},\nabla x_i\rangle\\
 &=2x_i-3(d+2)x_{d+1}^2x_i,\qquad 1\leq i\leq d.
\end{aligned}
\end{equation}
Reflections $x_i\mapsto-x_i$ show that both bilinear forms
$(\Delta u,\Delta v)_{L^2(\dS^d_+)}$ and $(-\Delta u,v)_{L^2(\dS^d_+)}$ are diagonal in the basis
$\{x_{d+1}^2,x_{d+1}^2x_1,\dots,x_{d+1}^2x_d\}$, that is for all $i,j\in\{1,2,\dots,d\}$ with $i\ne j$
\begin{equation}\label{eq:orth}
	\begin{aligned}
	\big(\Delta(x_{d+1}^2), \Delta(x_{d+1}^2x_i) \big)_{L^2(\dS^d_+)} &= 0,\\
	 \big(\Delta(x_{d+1}^2x_i), \Delta(x_{d+1}^2x_j) \big)_{L^2(\dS^d_+)} &= 0,\\
	 \big(\Delta(x_{d+1}^2x_i), x_{d+1}^2x_j \big)_{L^2(\dS^d_+)}&=0,\\
	 \big(\Delta(x_{d+1}^2), x_{d+1}^2x_j\big)_{L^2(\dS^d_+)} &=0.
\end{aligned}	
\end{equation}
Let us introduce the notation for the clamped buckling Rayleigh quotient
\[
	\sfR[u] := \frac{\displaystyle \int_{\dS^d_+} (\Delta u)^2\dd V}{\displaystyle\int_{\dS^d_+} (-\Delta u)u\dd V },\qquad u\in H^2_0(\dS^d_+)\sm\{0\}.
\]
This Rayleigh quotient is deduced from~\eqref{eq:buckling-minmax} upon integration by parts in the denominator.
Permutations among the first $d$ coordinates imply that the Rayleigh quotients evaluated on the last
$d$ functions $\{x_{d+1}^2x_j\}_{j=1}^d$ are
equal; \ie
\[
	\sfR[x_{d+1}^2x_i]= \sfR[x_{d+1}^2x_j],\qquad i,j\in\{1,2,\dots,d\}.
\]
We use the mapping $\sfE\colon L^1(\dS^d)\arr \dR$ defined by
\[
	\sfE[u] := \frac{\displaystyle \int_{\dS^d}u\dd V}{|\dS^d|}.
\]
We stress that every integrand below is even in $x_{d+1}$, and its
hemisphere mean equals its full-sphere mean.
For nonnegative integers $\{a_j\}_{j=1}^{d+1}$, we find
\begin{equation}
\begin{aligned}
 \sfE\left[\prod_{j=1}^{d+1}x_j^{2a_j}\right] &= \frac{\displaystyle \int_{\dR^{d+1}} e^{-|x|^2} \left(\prod_{j=1}^{d+1}x_j^{2a_j}\right)\dd x}{\displaystyle  \int_{\dR^{d+1}} e^{-|x|^2}|x|^{\sum_{j=1}^{d+1}2a_j} \dd x}\\
 &=\frac{\Gamma((d+1)/2)}{\Gamma((d+1)/2+\sum_{j=1}^{d+1} a_j)}
   \prod_{j=1}^{d+1}\frac{\Gamma(a_j+1/2)}{\Gamma(1/2)}.
\end{aligned}
 \label{eq:sphere-moments}
\end{equation}
In particular, for any $i\in\{1,2,\dots,d\}$
\begin{equation}\label{eq:E}
\begin{gathered}
 \sfE[x_{d+1}^2]=\frac1{d+1},\qquad
 \sfE[x_{d+1}^4]=\frac3{(d+1)(d+3)},\qquad
 \sfE[x_i^2]=\frac1{d+1},\\
 \sfE[x_{d+1}^2x_i^2]=\frac1{(d+1)(d+3)},\qquad
 \sfE[x_{d+1}^4x_i^2]=\frac3{(d+1)(d+3)(d+5)}.
\end{gathered}
\end{equation}
By \eqref{eq:sphere-t2} combined with~\eqref{eq:E} we obtain
\[
\begin{aligned}
 \sfE[x_{d+1}^2(-\Delta(x_{d+1}^2))]&=2(d+1)\sfE[x_{d+1}^4]-2\sfE[x_{d+1}^2]
       =\frac{4d}{(d+1)(d+3)},\\
 \sfE[(\Delta(x_{d+1}^2))^2]&=4(d+1)^2\sfE[x_{d+1}^4]-8(d+1)\sfE[x_{d+1}^2]+4=\frac{8d}{d+3}.
\end{aligned}
\]
Thus, we get
\begin{equation}\label{eq:R1}
	\sfR[x_{d+1}^2] = \frac{\sfE[(\Delta(x_{d+1}^2))^2]}{\sfE[x_{d+1}^2(-\Delta(x_{d+1}^2))]} = 2(d+1).
\end{equation}
%
Analogously, using \eqref{eq:sphere-t2} combined with~\eqref{eq:E} we get for $i\in\{1,2,\dots,d\}$
\[
\begin{aligned}
 \sfE[x_{d+1}^2x_i(-\Delta(x_{d+1}^2x_i))]&=3(d+2)\sfE[x_{d+1}^4x_i^2]-2\sfE[x_{d+1}^2x_i^2]\\
   &=\frac{7d+8}{(d+1)(d+3)(d+5)},\\
 \sfE[(\Delta(x_{d+1}^2x_i))^2]&=9(d+2)^2\sfE[x_{d+1}^4x_i^2]-12(d+2)\sfE[x_{d+1}^2x_i^2]+4\sfE[x_i^2]\\
 &=\frac{19d^2+56d+48}{(d+1)(d+3)(d+5)}.
\end{aligned}
\]
Thus, we obtain that  for all $i\in\{1,2,\dots,d\}$
\begin{equation}\label{eq:R2}
	\sfR[x_{d+1}^2x_i] = \frac{\sfE[(\Delta(x_{d+1}^2x_i))^2]}{\sfE[x_{d+1}^2x_i(-\Delta(x_{d+1}^2x_i))]}
	= \frac{19d^2+56d+48}{7d+8}= Q_d.
\end{equation}
Moreover, we observe that
\[
 Q_d-2(d+1)=\frac{5d^2+26d+32}{7d+8}>0.
\]
We conclude from~\eqref{eq:R1},~\eqref{eq:R2} combined with~\eqref{eq:orth} that
\[
	\max_{u\in E\setminus\{0\}}\sfR[u]=Q_d.
\]
The min-max principle gives the first inequality in \eqref{eq:hemisphere-bound};
subtracting $Q_d$ from $3(d+2)$ gives \eqref{eq:hemisphere-margin}.
\end{proof}
%

\subsection*{Acknowledgements} D.\ Buoso would like to thank M.\ Levitin and I.\ Polterovich for fruitful discussions on this manuscript and their manuscript \cite{LevitinPolterovich}. D.\ Buoso is a member of the Gruppo Nazionale per l'Analisi Matematica, la Probabilit\`a e le loro Applicazioni (GNAMPA) of the Istituto Naziona\-le di Alta Matematica (INdAM), and acknowledges financial support from the project ``Functional and spectral analysis for differential operators'' funded by GNAMPA INdAM.

\subsection*{AI usage disclosure} The authors acknowledge the use of AI for mathematical discussions and editorial assistance. In particular, the AI tool suggested the counterexamples in Section~\ref{sec:counterexamples}. All AI-assisted arguments and computations were checked and substantially rewritten by the authors, who take  full responsibility for their correctness.

\end{document}